\documentclass{amsart}

\usepackage[english]{babel}

\usepackage[letterpaper,top=2cm,bottom=2cm,left=3cm,right=3cm,marginparwidth=1.75cm]{geometry}

\usepackage{amsmath}
\usepackage{amsthm}
\usepackage{graphicx}
\usepackage[colorlinks=true, allcolors=blue]{hyperref}

\usepackage{setspace}
\usepackage{enumerate}

\newcommand{\ca}{\mathcal{A}}
\newcommand{\cb}{\mathcal{B}}
\newcommand{\cc}{\mathcal{C}}
\newcommand{\ci}{\mathcal{I}}
\newcommand{\cm}{\mathcal{M}}
\newcommand{\cn}{\mathcal{N}}

\newcommand{\ch}{H}
\newcommand{\I}{\mathcal{I}}
\newcommand{\M}{\mathcal{M}}
\newcommand{\C}{\mathcal{C}}
\newcommand{\cd}{\mathcal{D}}

\newtheorem{thm}{Theorem}[section]

\newtheorem{lem}[thm]{Lemma}
\newtheorem{obs}[thm]{Observation}
\newtheorem{observation}[thm]{Observation}
\newtheorem{definition}[thm]{Definition}
\newtheorem{claim}[thm]{Claim}  

\newtheorem{notation}[thm]{Notation}
\newtheorem{remark}[thm]{Remark}

\newcommand{\refT}[1]{Theorem~\ref{#1}}

\newcommand{\refL}[1]{Lemma~\ref{#1}}

\newcommand{\refS}[1]{Section~\ref{#1}}

\newcommand{\refOb}[1]{Observation~\ref{#1}}

\title{Coloring the intersection of two matroids}
\author{Eli Berger}
\address{Department of Mathematics, University of Haifa, Haifa 31905,  Israel}
\email{berger.haifa@gmail.com}
\author{He Guo}
\address{Department of Mathematics and Mathematical Statistics, Ume\r{a} University, Ume\r{a} 90187, Sweden}
\email{he.guo@umu.se}

\begin{document}
	\date{August 2024; Revised in January 2025}

	\maketitle
	
	\begin{abstract}
		A result~[The intersection of a matroid and a simplicial complex, {\em Trans. Amer. Math. Soc.} \textbf{358}] from 2006 of Aharoni and the first author of this paper states that for any two positive integers $p,q$, where $p$ divides $q$, if a matroid $\cm$ is $p$-colorable and a matroid $\cn$ is $q$-colorable then $\cm \cap \cn$ is $(p+q)$-colorable. In this paper we show that the assumption that~$p$ divides~$q$ is in fact redundant, and we also prove that $\cm \cap \cn$ is even $p+q$ list-colorable.
		
		The result uses topology and relies on a new parameter yielding a lower bound for the topological connectivity of the intersection of two matroids.
	\end{abstract}
	\section{Introduction}\label{sec:intro}
	
	A {\em hypergraph} is a pair $H=(V,E)$ where the \emph{vertex set}~$V$ is a finite set and the \emph{edge set}~$E$ is a set of subsets of~$V$.
	A set $X \subseteq V$ is called {\em independent} in $H$ if there is no $e \in E$ such that $e \subseteq X$. We denote by $\ci(H)$ the set of all independent sets in $H$. 
	
	A property of a set of the form $\cc = \ci(H)$ is that if $T \in \cc$ and $S \subseteq T$ then $S \in \cc$. For a finite set~$V$, a set~$\cc$ of subsets of~$V$ with such closed under taking subsets property is called an {\em (abstract simplicial) complex}, and $V$ is called the \emph{ground set} of the complex~$\cc$.  For convenience, in this note we assume that every element in the ground set of a complex is included in some set of the complex.
	
	\begin{definition}
		A complex~$\cm$ is called a matroid if the following hold:
		\begin{itemize}
			\item $\emptyset\in\cm$.
			\item (Independence augmentation axiom) If $S,T\in\cm$ and $|S|<|T|$, then there exists $v\in T\setminus S$ such that $S\cup\{v\}\in\cm$.
		\end{itemize}
		A set in~$\cm$ is called an independent set of the matroid~$\cm$.
	\end{definition}

	Two widely studied parameters in graph theory, chromatic number and list chromatic number of a graph~$G$, can be generalized to a complex~$\cc$, as defined in the following way (in the former case, $\cc=\ci(G)$).
	
	\begin{definition}
		Given a complex~$\cc$ on the ground set~$V$, the chromatic number~$\chi(\cc)$ of~$\cc$ is the minimum number of sets in~$\cc$ such that their union is~$V$. 
	\end{definition}
	\begin{definition}
		Given a complex~$\cc$ on the ground set~$V$, the list chromatic number~$\chi_\ell(\cc)$ of~$\cc$ is the minimum number~$k$ such that for any lists of colors $(L_v)_{v\in V}$ of size~$k$, there exists a choice function $f: V \rightarrow \cup_{v\in V} L_v$ such that $f(v)\in L_v$ for every~$v$ and $f^{-1}(c)\in\cc$ for every color~$c$.
	\end{definition}
	Setting $L_v=\{1,\dots,\chi_\ell(\cc)\}$ for each~$v\in V$ proves
	\[\chi(\cc)\le \chi_\ell(\cc).\]
	
	In~\cite{matcomp}, it is proved that for two matroids~$\cm$ and~$\cn$ on the same ground set, then 
	\[ \chi(\cm\cap\cn)\le 2\max(\chi(\cm),\chi(\cn)). \]
	And it is also proved that if $\chi(\cm)$ divides~$\chi(\cn)$, then
	\[ \chi(\cm\cap\cn)\le \chi(\cm)+\chi(\cn). \]
	In this paper, we extend these results.
	\begin{thm}
		For two matroids~$\cm$ and~$\cn$ on the same ground set,
		\[  \chi(\cm\cap\cn)\le \chi(\cm)+\chi(\cn). \]
	\end{thm}
	
	Together with a recent result in~\cite{intersectionofmatroids}, our proof also leads to the same upper bound on the list chromatic number.
	
	\begin{thm}\label{thm:chiellsum}
		For two matroids~$\cm$ and~$\cn$ on the same ground set,
		\[  \chi_\ell(\cm\cap\cn)\le \chi(\cm)+\chi(\cn). \]
	\end{thm}
	
	The structure of the paper is the following: in~\refS{sec:operators} we introduce some notation to state our results later. In~\refS{sec:top} we find a lower bound on the topological connectivity (see \refT{thm:colooporcontract}). In~\refS{sec:comb}, we introduce a combinatorial parameter. In~\refS{sec:topvscomb} we find a connection between the topological connectivity and the combinatorial parameter (see~\refT{thm:etanupq}). In~\refS{sec:proof}, we prove \refT{thm:chiellsum} by using the relations between (list) chromatic number, topological connectivity, and the combinatorial parameter.
	
	\section{Circuit representation of a matroid and some hypergraph operations}\label{sec:operators}
	
	The reversed operator for $\ci$ shown in~\refS{sec:intro} can be defined as follows.
	
	\begin{definition}
		Let $\cc$ be a complex on the ground set $V$. We write
		\[circ(\cc) = \{e \subseteq V : e \not\in \cc, \text{ and for every } x \in e,\; e \setminus \{x\} \in \cc \}.\]
	\end{definition}
	
	\begin{observation}
		\label{complexhypergraphduality}
		If $\cc$ is a complex on $V$ then $\cc = \ci((V,circ(\cc)))$.
	\end{observation}

	For a matroid~$\cm$ on the ground set~$V$, it is well-known that for the hypergraph $H=(V,E)$ with $E=circ(\cm)$,  the following properties hold:
	\begin{itemize}
		\item $\emptyset \not \in E$,
		\item there are no two distinct edges $C_1,C_2 \in E$ such that $C_1 \subseteq C_2$, and
		\item (circuit elimination property) for every two edges $C_1,C_2 \in E$ and every $u \in C_1 \cap C_2$ and $v \in C_1 \setminus C_2$ there exists $C_3 \in E$, such that $C_3 \subseteq C_1 \cup C_2$ and $v \in C_3$ and $u \not\in C_3$.
	\end{itemize}
	We call such~$H$ a \emph{circuit representation} of the matroid~$\cm$, and an element of~$E=circ(\cm)$ is called a \emph{circuit} of the matroid~$\cm$. On the other hand, it is well-known that for any hypergraph $H=(V,E)$ satisfying the above properties,~$\ci(H)$ is a matroid, of which~$H$ is a circuit representation.

	\begin{definition}
		Let $H = (V,E)$ be a hypergraph.
		For an edge $e \in E$, we write 
		\[H - e = (V,E \setminus \{e\}).\]
		For a set $X \subseteq V$, we write 
		\begin{align*}
			H[X] & = (X, \;\{e\in E \mid e\subseteq X \}),\\
			H / X &= (V \setminus X , 
			\{e \setminus X \mid e \in E,\; e\not\subseteq X\}),\\
			H \setminus X &= 
			(V \setminus X,\; \{e \in E \mid e \cap X = \emptyset \}),\\
			H \sim X &= 
			(V ,\; \{e \in E \mid e \cap X = \emptyset \}) .
		\end{align*}
		If $v \in V$ then we write $H \sim v = H \sim \{v\}$.
	\end{definition}
	Note that $H \setminus X$ and $H \sim X$ differ only by their vertex sets, and $H[X]=H\setminus (V\setminus X)$.
	
	When applying the operators~$[ \;\;]$,~$/$, and~$\sim$ to a matroid, we refer to the application of these operators on the circuit representation of the matroid. Formally, if~$\cm$ is a matroid on the ground set~$V$, whose circuit representation is $H = (V,circ(\cm))$, and if $X \subseteq V$, then we define 
	\[\cm[X]=\ci(H[X]),\quad  \cm / X = \ci(H / X),\quad\text{and}\quad \cm\sim X =\ci(H\sim X),\]
	and if $v \in V$ then we define
	\[\cm \sim v = \ci(H \sim v).\]
	
	Note that $\cm[X]$, $\cm / X$, and $\M\sim X$ in the above definitions are matroids, and they coincide with the usual definitions of the restriction of~$\cm$ to~$X$, the contraction of~$X$ from~$\cm$, and the join (direct sum) $2^X*\cm[V\setminus X]$ (see \refS{sec:top} for the definition), respectively.

	\section{The topological parameter $\eta$}\label{sec:top}
	
	Let $\cc$ be an abstract simplicial complex. 
	Assuming we fix some ring $R$, we can apply homology theory on $\cc$. We write \emph{(homological) connectivity} $\eta(\cc)$ for the minimal value of $k$ such that the reduced homology $\tilde{H}_{k-1}(\cc,R)$ does not vanish. If $\cc$ is the empty complex then we write $\eta(\cc) = 0$ and if all reduced homology groups of $\cc$ vanish then we write $\eta(\cc) = \infty$. See, e.g.,~\cite[Section 2]{matcomp}, for the geometric meaning of the connectivity.
	
	For two complexes~$\cc,\cd$ on disjoint sets, the \emph{join}  $\cc*\cd$ is $\{S\cup T\mid S\in\cc, T\in\cd\}$.
	\begin{thm}\cite{matcomp}\label{thm:etajoin}
		Let $\cc$ and $\cd$ be complexes on disjoint sets. Then
		\[ \eta(\cc*\cd)\ge\eta(\cc)+\eta(\cd). \]
	\end{thm}
	
	For two complexes on the same set, we can prove the following inequalities using the Mayer–Vietoris sequence:
	
	\begin{thm}
		Let $\ca$ and $\cb$ be two abstract simplicial complexes on the same set $V$. Then
		
		\begin{enumerate}
			\label{unionintersection}
			\item\label{eq:1} $\eta(\ca \cup \cb) \geq \min\Big(\eta(\ca),\; \eta(\cb),\; \eta(\ca \cap \cb) + 1\Big)$.
			\item $\eta(\ca \cap \cb) \geq \min\Big(\eta(\ca),\; \eta(\cb), \;\eta(\ca \cup \cb) - 1\Big)$.
			\item\label{eq:3} $\eta(\ca) \geq \min\Big(\eta(\ca\cup \cb), \; \eta(\ca\cap \cb)\Big)$.
		\end{enumerate}
	\end{thm}
	
	Using inequality~\eqref{eq:1} of \refT{unionintersection} for $\ca=\ci(H\setminus\{v\})$ and~$\cb=2^{\{v\}}*\ci\Big(H\setminus\big(\{v\}\cup N_H(v)\big)\Big)$, where $2^S$ is the power set of the set~$S$ and $N_H(v)$ is the set of neighbors of~$v$ in~$H$, we can deduce the following:
	\begin{thm}
		\label{thm:vertexlink}
		Let $H = (V,E)$ be a hypergraph and let $v$ be a vertex of $H$ such that $\{v\} \not\in E$. Then 
		\[\eta(\ci(H)) \geq \min\Big(\eta\big(\ci(H \setminus \{v\})\big),\;
		\eta\big(\ci(H / \{v\})\big) + 1\Big).\]
	\end{thm}
	
	Using inequality~\eqref{eq:3} of \refT{unionintersection} for $\ca=\ci(H)$ and $\cb=2^e*\ci(H/e)$  we can also deduce the following (see~\cite[Section 8.3]{intersectionofmatroids} for the details of the proof):
	\begin{thm}
		\label{thm:gamestep}
		Let $H$ be a hypergraph and let $e$ be an edge of $H$ which does not contain any other edge. Then 
		\[\eta(\ci(H)) \geq \min\Big(\eta\big(\ci(H-e)\big),\;
		\eta\big(\ci(H / e)\big) + |e|-1\Big).\]
	\end{thm}
	
	\refT{thm:gamestep} was proved in \cite{meshulam} for the case that $H$ is a graph, but the same proof holds for general hypergraphs as well. While \refT{thm:gamestep} is extensively used for graphs, its use for hypergraphs is less common so far. It is used implicitly in \cite{digraphs}.

	Repeatedly applying \refT{thm:vertexlink} and \refT{thm:gamestep} gives a very powerful tool for obtaining lower bounds for $\eta$.

	\begin{thm}
		\label{thm:colooporcontract}
		Let $\cm_1, \ldots, \cm_k$ be matroids on the ground set $V$, and let $v 
		\in V$. Then either 
		\[\eta(\cm_1 \cap \ldots \cap \cm_k) \geq 
		\eta\Big((\cm_1 \sim v) \cap \cm_2 \cap \ldots \cap \cm_k\Big)\] 
		or there exists $C$ such that
		\begin{enumerate}[(i)]
			\item $v \in C$.
			\item $C$ is a circuit of $\cm_1$.
			\item $C \in \bigcap_{i=2}^k \cm_i$.
			\item $\eta(\cm_1 \cap \ldots \cap \cm_k) \geq 
			\eta\Big((\cm_1  / C) \cap \ldots \cap (\cm_k / C)\Big) + |C| - 1$.
		\end{enumerate}
	\end{thm}
	
	\begin{proof}
		Let $\ch_i=(V,E_i)=(V,circ(\M_i))$ be the circuit representation of~$\M_i$ for $i=1,\dots, k$.
		For $I\subseteq \{1,\dots, k\}$, let $\cup_{i\in I}\ch_i=(V,\cup_{i\in I}E_i)$ and let
		$\ch=\cup_{i=1}^k\ch_i$. So $\I(\ch)=\cap_{i=1}^k\M_i$.
		
		Let $C_1,\dots,C_t$ be all the circuits of~$\M_1$ satisfying that $v\in C_j$ and $C_j\in \cap_{i=2}^k\M_i$ for each $1\le j\le t$.
		
		\begin{claim}
			\begin{gather}
				\I(\ch-C_1-\cdots -C_t)= (\M_1\sim v)\cap\M_2\cap\cdots\cap\M_k,   \label{eq:eq1}\\
				\I\big(  (\ch-C_1-\cdots-C_{j-1})/C_j  \big)   =\I(\ch/C_j) \quad\text{for each $1\le j\le t$}, \label{eq:eq2} \\
				\I(\ch/C_j)= (\M_1/C_j)\cap\cdots\cap(\M_k/C_j) \quad\text{for each $1\le j\le t$}.\label{eq:eq3} 
			\end{gather} 
		\end{claim}
		Suppose the claim is true, then applying~\refT{thm:gamestep} repeatedly, where each $C_j$ in turn takes the role of~$e$ in the theorem, yields either by~\eqref{eq:eq1} that
		\[  \eta(\I(\ch))\ge \eta\Big(\I(\ch-C_1-\cdots -C_t)\Big)=\eta\Big( (\M_1\sim v)\cap\M_2\cap\cdots\cap\M_k\Big),  \]
		or for some $1\le j\le t$ by~\eqref{eq:eq2} and~\eqref{eq:eq3} that
		\begin{align*}
			\eta(\I(\ch))&\ge \eta\Big(\I\big(  (\ch-C_1-\cdots-C_{j-1})/C_j  \big)    \Big) +|C_j|-1\\
			&=\eta\Big(\I(\ch/C_j) \Big)+|C_j|-1\\
			&=\eta\Big( (\M_1/C_j)\cap\cdots\cap(\M_k/C_j) \Big)+|C_j|-1,
		\end{align*}    
		in which case setting~$C=C_j$ completes the proof.
	\end{proof}
	\begin{proof}[Proof of the claim]
		To verify~\eqref{eq:eq1}, for~$S$ in the LHS, since for each $1\le j\le t$, $C_j\in\cap_{i=2}^k\cm_i$, which implies that~$C_j$ is not an edge of $\cup_{i=2}^k\ch_i$, then~$S$ does not contain any edge of $\cup_{i=2}^k\ch_i$, therefore $S\in\cap_{i=2}^k\M_i$. And $S\setminus\{v\}$ does not contain any edge $f\in\ch_1$ satisfying $f\subseteq V\setminus\{v\}$, since such $f\neq C_j$ for every $1\le j\le t$. Therefore $S\setminus\{v\}\in\M_1[V\setminus\{v\}]$ and $S\in \M_1\sim v$. Thus the LHS is contained in the RHS.
		For $T$ in the RHS, $T\in\cap_{i=2}^k\M_i$ implies that~$T$ is independent in $\cup_{i=2}^k\ch_i$. And $T$ does not contain any edge of~$\ch_1[V\setminus\{v\}]$.  Furthermore, we claim that~$T$ does not contain any $C\in\ch_1$ such that $v\in C$ and $C\neq C_j$ for every $1\le j\le t$: suppose not, then such $C$ must contain an edge of some~$\ch_i$ for $2\le i\le k$ as $C_1,\dots,C_t$ are all the edges of~$\ch_1$ including~$v$ and containing no edge of $\cup_{i=2}^t\ch_{i}$, which contradicts to the assumption $T\in \M_i$. Therefore $T\in\I(\ch-C_1-\cdots -C_t)$ and then  the RHS is contained in the LHS.
		
		To verify~\eqref{eq:eq2}, the RHS is contained in the LHS, since $(\ch-C_1-\cdots-C_{j-1})/C_j$ and~$\ch/C_j$ have the same vertex set and the edge set of the former is contained in that of the latter. On the other hand, for any $S$ in the LHS, we claim that $S$ does not contain $C_\ell\setminus C_j$ for any $1\le \ell\le j-1$: suppose not, i.e., $C_\ell\setminus C_j\subseteq S$ for some $1\le \ell\le j-1$. Since $v\in C_j\cap C_{\ell}$, then by the circuit elimination property, there exists a circuit $C'$ of~$\M_1$ such that $C'\subseteq (C_j\cup C_\ell)\setminus\{v\}$. Then $S$ contains $C'\setminus C_j$, but~$C'$ is an edge of~$H$ different from any~$C_\ell$  for $1\le \ell\le j-1$ (since $v\not\in C'$), which contradicts with the assumption that $S\in \I\big((\ch-C_1-\dots-C_{j-1})/C_j\big)$. Therefore the claim holds and $S\in\I(H/C_j)$, which proves that the LHS is contained in the RHS.
		
		To verify~\eqref{eq:eq3}, since $C_j\in\cap_{i=2}^k\M_i$, no edge of $\cup_{i=2}^k\ch_i$ is contained in $C_j$. And since~$C_j\in circ(\M_1)$, no edge of~$\ch_1$ other than~$C_j$ itself is contained in~$C_j$. Therefore
		\begin{equation}\label{eq:4}
			\ch/C_j=(\cup_{i=1}^k\ch_i)/C_j = \cup_{i=1}^k(\ch_i/C_j). 
		\end{equation}
		By definition~$\I(\ch_i/C_j)=\M_i/C_j$ for each~$i=1,\dots, k$, therefore together with~\eqref{eq:4},
		$\I(\ch/C_j)= \cap_{i=1}^k\I(\ch_i/C_j) = \cap_{i=1}^k(\M_i/C_j)$.
	\end{proof}

	\section{The combinatorial parameter $\nu_{p,q}(\cm,\cn)$}\label{sec:comb}
	
	For sets $A_1, \ldots, A_k$ and an element $v$, we write $\#v(A_1, \ldots, A_k)$ for the number of sets among $A_1, \ldots, A_k$ to which $v$ belongs. For two positive integers $p,q$ and two matroids $\cm,\cn$ on the same ground set $V$ we define
	
	\[
	\nu_{p,q}(\cm,\cn) = \max_{\substack{A_1, \ldots, A_p \in \cm \\ B_1, \ldots, B_q \in \cn }} \, \sum_{v \in V} \min \big(\#v(A_1, \ldots, A_p),\;\#v(B_1, \ldots, B_q)\big).
	\]
	
	Note that $\nu_{1,1}(\cm,\cn)$ is the size of the largest set which is independent in both $\cm$ and $\cn$. 
	
	\begin{obs}\label{ob:monotone}
		The parameter $\nu$ is monotone in $p$, $q$, $\cm$ and $\cn$, i.e., if $p,q,p',q'$ are positive integers and $\cm, \cn, \cm', \cn'$ are matroids such that $p \leq p'$, $q \leq q'$, $\cm \subseteq \cm'$ and $\cn \subseteq \cn'$ then $\nu_{p,q}(\cm,\cn)\leq\nu_{p',q'}(\cm',\cn')$
	\end{obs}
	By the closed down property of~$\cm$ and~$\cn$ and the double-counting argument, we can get the following observation.
	\begin{obs}
		For any two positive integers $p,q$ and two matroids $\cm,\cn$ on the same ground set $V$, there exist sets 
		$A_1, \ldots, A_p \in \cm$ and $B_1, \ldots, B_q \in \cn$ such that $\sum_{i=1}^p |A_i| = \sum_{j=1}^q |B_j| = \nu_{p,q}(\cm,\cn)$ and every $v \in V$ satisfies $\#v(A_1, \ldots, A_p) = \#v(B_1, \ldots, B_q)$.
	\end{obs}
	
	\begin{lem}\label{lemma:nuqqqnu11}
		If $\cm,\cn$ are matroids on a common ground set, then for every positive integer $q$
		$$
		\nu_{1,1}(\cm,\cn) \geq \lceil \frac{\nu_{q,q}(\cm,\cn)}{q} \rceil.
		$$
	\end{lem}
	
	\begin{proof}
		
		Let $V$ be the common ground set.
		By Edmonds' matroid intersection theorem there exist sets $I,V_1,V_2 \subseteq V$ such that $V = V_1 \cup V_2$ and $V_1 \cap V_2 = \emptyset$ and $I \in \cm \cap \cn$ and $I \cap V_1$ is the largest among subsets of $V_1$ which are independent in $\cm$, and  $I \cap V_2$ is the largest among subsets of $V_2$ which are independent in $\cn$. 
		
		Now for every $A_1, \ldots, A_q \in \cm$ and $B_1, \ldots, B_q \in \cn$ we have
		\begin{align*}
			& \sum_{v \in V} \min \big(\#v(A_1, \ldots, A_q), \;\#v(B_1, \ldots, B_q)\big) \\
			= &\sum_{v \in V_1} \min \big(\#v(A_1, \ldots, A_q),\; \#v(B_1, \ldots, B_q)\big) +
			\sum_{v \in V_2} \min \big(\#v(A_1, \ldots, A_q),\;\#v(B_1, \ldots, B_q)\big) \\
			\leq&
			\sum_{v \in V_1} \#v(A_1 \cap V_1, \ldots, A_q \cap V_1 ) +
			\sum_{v \in V_2} 
			\#v(B_1 \cap V_2, \ldots, B_q \cap V_2) \\
			\leq& q\cdot |I \cap V_1| + q\cdot |I \cap V_2| = q|I| \leq q\nu_{1,1}(\cm,\cn),
		\end{align*}
		which completes the proof.
	\end{proof}
	
	\begin{lem}
		\label{dangling}
		Let $\cm,\cn$ be two matroids on the same ground set $V$ such that $\nu_{1,1}(\cm,\cn)>0$.
		Let $p,q$ be positive integers with $p \leq q$.
		Then there exist sets $X_1, \ldots X_p \in \cm$ and $Y_1, \ldots, Y_q \in \cn$ and an element $z \in V$ with the following properties:
		\begin{itemize}
			\item  Every $v \in V \setminus \{z\}$ satisfies $\#v(X_1, \ldots, X_p) = \#v(Y_1, \ldots, Y_q)$,
			\item  $\#z(X_1, \ldots, X_p) = p$,
			\item  $\sum_{i=1}^q |Y_i| = \nu_{p,q}(\cm,\cn)$.
		\end{itemize}
	\end{lem}
	
	\begin{proof}
		Write $r = \nu_{1,1}(\cm,\cn)$.
		Let $Z \in \cm \cap \cn$ have size $r$, and let $A_1, \ldots, A_p \in \cm$ and $B_1, \ldots, B_q \in \cn$ be chosen such that 
		\begin{enumerate}[(a)]
			\item\label{eq:condition1} $\sum_{v \in V} \min \big(\#v(A_1, \ldots, A_p),\; \#v(B_1, \ldots, B_q)\big) = \nu_{p,q}(\cm,\cn)-1$,
			\item\label{eq:condition2} Subject to the above condition 
			$\sum_{v \in Z} \min \big(\#v(A_1, \ldots, A_p),\; \#v(B_1, \ldots, B_q)\big)$ is maximal,
			\item\label{eq:condition3} Subject to the above two conditions $\sum_{i=1}^p |A_i| + \sum_{j=1}^q |B_j|$ is minimal.
		\end{enumerate}
		Note that $\sum_{i=1}^p |A_i| + \sum_{j=1}^q |B_j|$ is minimal in condition~\eqref{eq:condition3} guarantees that for every $v\in V$,
		\begin{equation}\label{eq:observationequal}
			\#v(A_1, \ldots, A_p)=\#v(B_1, \ldots, B_q),
		\end{equation}
		since if $\#v(A_1, \ldots, A_p)>\#v(B_1, \ldots, B_q)$ for some~$v\in V$, we can remove~$v$ from some of~$A_i$ including it, which does not violate condition~\eqref{eq:condition1} or~\eqref{eq:condition2}, but has a smaller total size, a contradiction. Similarly, we can get a contradiction if $\#v(A_1, \ldots, A_p)<\#v(B_1, \ldots, B_q)$ for some~$v\in V$.
		
		Therefore we have
		\[\sum_{j=1}^q |B_j|=\sum_{v\in V}\#v(B_1, \ldots, B_q) = \nu_{p,q}(\cm,\cn)-1 < \nu_{p,q}(\cm,\cn)  \underset{\refOb{ob:monotone}}{\le} \nu_{q,q}(\cm,\cn)  \underset{\refL{lemma:nuqqqnu11}}{\le} qr,\]
		which implies at least one set among 
		$B_1, \ldots, B_q$ has size less than $r$.
		Without loss of generality $|B_q|< r$.
		By the independence augmentation axiom, this implies that for some $z \in Z\setminus B_q$ we have 
		$B'_q = B_q \cup \{z\} \in \cn$.
		
		We now claim that for every $i \in \{1, \ldots, p\}$ we have $A_i \cup \{z\} \in \cm$. Suppose not, then, say, $A_p \cup \{z\} \not\in \cm$, so it contains some circuit~$C$ of~$\cm$ and $z\in C$. Take some $c \in C \setminus Z$. Then~$c\in A_p$.
		By the independence augmentation axiom, $C\setminus\{c\}\in\cm$ can be extended to a size~$|A_p|$ set $A'_p = (A_p \cup \{z\}) \setminus \{c\} \in \cm$. 
		Since
		\[ \#c(A_1,\dots,A_{p-1},A_p')=\#c(A_1,\dots,A_{p-1},A_p)-1  \]
		and
		\begin{gather*}
			\min \Big(\#z(A_1, \ldots, A_{p-1}, A'_p),\; \#z(B_1, \ldots, B_{q-1}, B'_q)\Big)\\
			= \min \Big(\#z(A_1, \ldots, A_p),\; \#z(B_1, \ldots, B_q)\Big) +1, 
		\end{gather*}
		then by~\eqref{eq:observationequal} we now have 
		\begin{align*}
			&\sum_{v \in V} \min \Big(\#v(A_1, \ldots, A_{p-1}, A'_p),\; \#v(B_1, \ldots, B_{q-1}, B'_q)\Big) \\
			=&\sum_{v \in V} \min \Big(\#v(A_1, \ldots, A_{p-1}, A_p),\; \#v(B_1, \ldots, B_{q-1}, B_q) \Big)= \nu_{p,q}(\cm,\cn)-1
		\end{align*}
		and 
		\begin{gather*}
			\sum_{v \in Z} \min \Big(\#v(A_1, \ldots, A_{p-1}, A'_p),\; \#v(B_1, \ldots, B_{q-1}, B'_q)\Big)\\
			=\sum_{v \in Z} \min \Big(\#v(A_1, \ldots, A_p),\; \#v(B_1, \ldots, B_q)\Big) + 1
		\end{gather*}
		contradicting the way~\eqref{eq:condition2} in which $A_1, \ldots, A_q$ and $B_1, \ldots, B_q$ were chosen. This finishes the proof of the claim.  
		
		Now the sets $X_i = A_i \cup \{z\}$ (for $i=1, \ldots, p$) and $Y_j = B_j$ (for $j = 1, \ldots, q-1$) and $Y_q = B_q \cup \{z\}$ 
		satisfy the required conditions of the lemma.
	\end{proof}
	
	\section{Relation between the parameters}\label{sec:topvscomb}
	
	\begin{thm}\label{thm:etanupq}
		Let $p,q$ be two positive integers and let $\cm, \cn$ be two matroids on the same ground set. Then 
		\[\eta(\cm \cap \cn) \geq \frac{\nu_{p,q}(\cm,\cn)}{p+q} \; .\]
	\end{thm}

	\begin{proof}
		The proof is by induction on the size of the common ground set~$V$.
		When $V=\emptyset$, $\eta(\cm\cap\cn)=\nu_{p,q}(\cm,\cn)=0$. The statement is true.
		
		Next we turn to the case $|V|\ge 1$. We assume without loss of generality $p \leq q$ and by Lemma \ref{dangling}, there exist sets $X_1, \ldots X_p \in \cm$ and $Y_1, \ldots, Y_q \in \cn$ and an element $z \in V$ with the following properties:
		\begin{itemize}
			\item Every $v \in V \setminus \{z\}$ satisfies $\#v(X_1, \ldots, X_p) = \#v(Y_1, \ldots, Y_q)$,
			\item $\#z(X_1, \ldots, X_p) = p$,
			\item $\sum_{i=1}^q |Y_i| = \nu_{p,q}(\cm,\cn)$.
		\end{itemize}

		We apply \refT{thm:colooporcontract} with $\cm$ taking the role of $\cm_1$. Then either
		\begin{equation}\label{eq:thefirstcase}
			\eta(\cm \cap \cn) \geq \eta\Big((\cm \sim z) \cap \cn\Big)
		\end{equation}
		or there is some circuit~$C$ in~$\cm$ such that $z\in C$, $C\in\cn$, and
		\begin{equation}\label{eq:thesecondcase}
			\eta(\cm \cap \cn) \geq \eta\Big((\cm  / C) \cap (\cn / C)\Big) + |C| - 1.
		\end{equation}
		If~\eqref{eq:thesecondcase} occurs,
		write $I = C \setminus \{z\}$ and $s = |I| = |C|-1$.
		Let $A_1, \ldots, A_p \in \cm$ and $B_1, \ldots, B_q \in \cn$ satisfy
		$\nu_{p,q}(\cm,\cn) = \sum_{u \in V} \min \Big(\#u(A_1, \ldots, A_p),\;\#u(B_1, \ldots, B_q)\Big)$.

		We can find sets $S_1, \ldots S_p, T_1, \ldots, T_q$, each of size~$s$, such that $A_i \setminus S_i \in \cm / C$ for all $i \in \{1, \ldots, p\}$ and  $(B_j \setminus \{z\}) \setminus T_j \in \cn / C$ for all $j \in \{1, \ldots, q\}$. In detail, to construct~$S_i$, 
		since $C$ is a circuit in~$\cm$, $C\not\subseteq A_i$ and then $|A_i\cap C|<|C|=s+1$,
		we take $A_i\cap C$ and add any $s-|A_i\cap C|$ other elements. To construct~$T_i$, we take $(B_j\setminus\{z\})\cap C$ and add any $s-|(B_j\setminus\{z\})\cap C|$ other elements.
		
		Write $A'_i = A_i \setminus S_i$, which is in~$\cm/C$, for all $i \in \{1, \ldots, p\}$, and  $B'_j = (B_j \setminus \{z\}) \setminus T_j$, which is in~$\cn/C$, for all $j \in \{1, \ldots, q\}$.
		We now claim that for each $u \in V$ we have 
		\begin{align*}
			&\min \Big(\#u(A_1, \ldots, A_p),\; \#u(B_1, \ldots, B_q)\Big)\\
			\le & 
			\min \Big(\#u(A'_1, \ldots, A'_p),\; \#u(B'_1, \ldots, B'_q)\Big)
			+ \#u(S_1, \ldots, S_p, T_1, \ldots, T_q).
		\end{align*}

		Indeed, when $u \neq z$ this is trivial, since $u\in A_i$ if and only if $u\in A_i'$ or~$S_i$, and $u\in B_j$ if and only if $u\in B_j'$ or~$T_j$; and for $u=z$ it follows from that fact that whenever $z \in A_i$ we must have also $z \in S_i$ so that
		\[ \#z(A_1, \ldots, A_p)\le \#z(S_1, \ldots, S_p, T_1, \ldots, T_q).    \]
		
		We thus have
		\begin{align*}
			\nu_{p,q}(\cm,\cn) &= \sum_{u \in V} \min \Big(\#u(A_1, \ldots, A_p),\;\#u(B_1, \ldots, B_q)\Big)\\
			&\le \sum_{u\in V}\Bigg(  \min \Big(\#u(A'_1, \ldots, A'_p), \;\#u(B'_1, \ldots, B'_q)\Big)
			+ \#u(S_1, \ldots, S_p, T_1, \ldots, T_q)   \Bigg)\\
			&\le \nu_{p,q}(\cm / C, \cn / C)+(p+q)s
		\end{align*}
		and by the induction hypothesis 
		\[
		\eta(\cm \cap \cn) \geq \eta\Big((\cm / C) \cap (\cn / C)\Big) + s \geq
		\frac{\nu_{p,q}(\cm / C, \cn / C)}{p+q} + s \geq \frac{\nu_{p,q}(\cm, \cn)}{p+q}.
		\]

		If~\eqref{eq:thefirstcase} occurs, applying~\refT{thm:colooporcontract} again to $\cm\sim z$ and $\cn$ with $\cn$ taking the role of~$\cm_1$, either 
		\[\eta(\cm\cap\cn)\ge\eta\Big((\cm\sim z)\cap\cn\Big)\ge \eta\Big((\cm\sim z)\cap(\cn\sim z)\Big)\mathop{\ge}_{\text{\refT{thm:etajoin}}} \eta(2^{\{z\}})=\infty,\]
		in which case we have  $\eta(\cm\cap\cn)\ge \frac{\nu_{p+q}(\cm,\cn)}{p+q}$,
		or
		there exists a circuit~$D$ in~$\cn$ such that $z\in D$ and~$D\in\cm\sim z$ and 
		\[\eta(\cm\cap\cn)\ge \eta\Big((\cm\sim z)\cap \cn\Big)\ge
		\eta\Big( \big((\cm\sim z)/D \big)\cap\big(\cn/D\big) \Big)+|D|-1.\]
		In the last case, similar as above, we can find sets $S_1',\dots, S_p',T_1',\dots, T_q'$, each of size~$t=|D|-1$, such that for each $1\le i\le p$,
		$X_i':=(X_i\setminus\{z\})\setminus S_i'\in (\cm\sim z)/D$ using the fact that $z\in X_i$, and for each $1\le j\le q$, $Y_j':=Y_j\setminus T_j'\in\cn/D$. Thus
		\begin{align*}
			\#u(Y_1,\dots,Y_q)\le \min\Big( \#u(X_1',\dots,X_p'),\; \#u(Y_1',\dots,Y_q')  \Big) + \#u(S_1',\dots,S_p',T_1',\dots,T_q'),  
		\end{align*}
		since for $u\neq z$,  $\#u(Y_1,\dots,Y_q)=\#u(X_1,\dots, X_p)$, and for $u=z$, $z\in Y_j$ implies $z\in T_j'$.
		Therefore
		\begin{align*}
			\nu_{p,q}(\cm,\cn)&= \sum_{i=1}^q |Y_i| =\sum_{u\in V}\#u(Y_1,\dots,Y_q)\\ 
			&\le \sum_{u\in V}\Bigg( \min\Big( \#u(X_1',\dots,X_p'),\; \#u(Y_1',\dots,Y_q')  \Big) + \#u(S_1',\dots,S_p',T_1',\dots,T_q') \Bigg)\\
			&\le \nu_{p,q}\Big(  (\cm\sim z)/D,\cn/D   \Big)+(p+q)t,
		\end{align*}
		and by the induction hypothesis,
		\[\eta(\cm\cap\cn)\ge 
		\eta\Big( \big((\cm\sim z)/D \big)\cap\big(\cn/D\big) \Big)+t \ge \frac{\nu_{p,q}\Big(  (\cm\sim z)/D,\cn/D   \Big)}{p+q}+t\ge\frac{\nu_{p,q}(\cm,\cn)}{p+q},  \]
		which completes the proof.
	\end{proof}

	\section{Proof of~\refT{thm:chiellsum}}\label{sec:proof}
	\begin{notation}
		For a complex~$\C$ on the ground set~$V$, let $\Delta_{\eta}(\C)=\max_{\emptyset\neq S\subseteq V(\C)}\frac{|S|}{\eta(\C[S])}$.
	\end{notation}
	Applying~Theorem 4.2 in~\cite{matcomp}, for a complex~$\cc$, it is proved in Corollary 8.6 of~\cite{matcomp} that $\chi(\cc)\le \lceil\Delta_\eta(\cc)\rceil$. In~\cite{intersectionofmatroids}, this bound is extended to the list chromatic number.
	
	\begin{thm}\label{thm:Deltaeta}
		For a complex~$\C$, $\chi_\ell(\C)\le \lceil\Delta_{\eta}(\C)\rceil$.
	\end{thm}
	
	\begin{proof}[Proof of~\refT{thm:chiellsum}]
		Let~$V$ be the common ground set of~$\cm$ and~$\cn$ and let $\C=\cm\cap\cn$.
		
		Let $p=\chi(\cm)$ and $q=\chi(\cn)$.    Let $A_1,\dots,A_p\in \cm$ satisfying that $\cup_{i=1}^{p}A_i=V$ and let $B_1,\dots, B_{q}\in\cn$ satisfying that $\cup_{j=1}^{q}B_j=V$. Then
		\[ \min\Big(\#v(A_1,\dots,A_p), \#v(B_1,\dots,B_q)  \Big)\ge 1  \]
		for every $v\in V$, which implies that $\nu_{p,q}(\cm,\cn)\ge |V|$.
		Thus~\refT{thm:etanupq} implies that $\eta(\C)\ge \frac{|V|}{p+q}$, which is equivalent to
		\[   \frac{|V|}{\eta(\C)}\le p+q. \]
		Noting that $\chi(\cm[S])\le \chi(\cm)$ and $\chi(\cn[S])\le \chi(\cn)$ for every $S\subseteq V$, the above argument works for any non-empty subset~$S$ of~$V$, therefore we have
		\begin{equation}\label{eq:ineqDelta}
			\Delta_\eta(\C)\le \chi(\cm)+\chi(\cn), 
		\end{equation}
		which together with~\refT{thm:Deltaeta} completes the proof.
	\end{proof}
	\begin{remark}
		The bound $\Delta_\eta(\cm\cap\cn)\le \chi(\cm)+\chi(\cn)$ in~\eqref{eq:ineqDelta} is tight.
	\end{remark}
	\begin{proof}
		Consider the 4-cycle on $\{1,2,3,4\}$ whose edges\footnote{We use $uv$ as a shorthand for $\{u,v\}$.} are $12,23,34,41$. Then we blow up each of $12,34$ by~$p$ (parallel) edges and blow up each of $23,41$ by~$q$ (parallel) edges. Let the resulting graph be~$G$. 
		Next we define two partition matroids~$\cm$ and~$\cn$ on the ground set~$E(G)$. The parts of~$\cm$ are $\Gamma_G(1)$ (all the edges of~$G$ incident with vertex 1) and $\Gamma_G(3)$, which satisfies
		$\chi(\cm)=p+q$. The parts of~$\cn$ are $\Gamma_G(2)$ and $\Gamma_G(4)$, which satisfies $\chi(\cn)=p+q$. 
		Let~$\cc$ be the matching complex of~$G$, i.e., the collection of all the matchings in~$G$. Then~$\cc$ is the intersection of the two partition matroids~$\cm$ and~$\cn$. Since the matching complex has two connected components, $\eta(\cc)=1$ so that
		\[\Delta_\eta(\cm\cap\cn)=|E(G)|=2(p+q)=\chi(\cm)+\chi(\cn),\]
		where the equality holds.
	\end{proof}

	\noindent\textbf{Acknowledgment:} We thank the anonymous reviewers for their careful reading and helpful comments. Research of the second author is supported by the Kempe Foundation grant JCSMK23-0055.
	
	\small
	\bibliographystyle{abbrv}
	
	\normalsize
\end{document}